\documentclass{daj}

\usepackage{latexsym,amssymb,amsfonts,amsmath,mathrsfs,float, amsthm}

  \newcommand{\Exp}{{\mathbb{E}}}

  \DeclareMathOperator{\cb}{cb}
  \DeclareMathOperator{\jcb}{jcb}
  \DeclareMathOperator{\sym}{sym}

  \newcommand{\R}{\mathbb{R}} % reals
  \newcommand{\C}{\mathbb{C}} % complex numbers
  \newcommand{\N}{\mathbb{N}} % natural numbers
  \newcommand{\Z}{\mathbb{Z}} % integers
  \newcommand{\T}{\mathbb T} % circle
  \newcommand{\pmset}[1]{\{-1,1\}^{#1}} % hypercube in +-1 basis
  \newcommand{\Id}{\ensuremath{\mathop{\rm Id}\nolimits}}

 \newcommand{\Zp}{\mathbb Z_p} 

  \newcommand{\st}{:\,} % "such that" to define sets

  \newcommand{\A}{\mathcal A}
  \newcommand{\B}{\mathcal B}
  
  \newcommand{\HS}{\mathcal H}
   
  \newcommand{\beq}{\begin{equation}}
  \newcommand{\eeq}{\end{equation}}
  \newcommand{\beqn}{\begin{equation*}}
  \newcommand{\eeqn}{\end{equation*}}
  \newcommand{\beqr}{\begin{eqnarray}}
  \newcommand{\eeqr}{\end{eqnarray}}
  \newcommand{\beqrn}{\begin{eqnarray*}}
  \newcommand{\eeqrn}{\end{eqnarray*}}
  \newcommand{\bmline}{\begin{multline}}
  \newcommand{\emline}{\end{multline}}
  \newcommand{\bmlinen}{\begin{multline*}}
  \newcommand{\emlinen}{\end{multline*}}
  
  \theoremstyle{plain}
  \newtheorem{theorem}{Theorem}[section]
  \newtheorem{lemma}[theorem]{Lemma}
  \newtheorem{proposition}[theorem]{Proposition}
  
  \newtheorem{corollary}[theorem]{Corollary}
  
  \theoremstyle{definition}
  \newtheorem{definition}[theorem]{Definition}

  \theoremstyle{remark}
  \newtheorem{remark}[theorem]{Remark}
  
  \renewenvironment{proof}[1][]{
    	\begin{trivlist}
     	\item[\hspace{\labelsep}{\em\noindent Proof#1:\/}]}
     	{{\hfill$\Box$}
    	\end{trivlist}
  }

\dajAUTHORdetails{%
  title = {Failure of the trilinear operator space Grothendieck Theorem}, %% please capitalize all significant words
  author = {Jop Bri\"{e}t and Carlos Palazuelos},
  plaintextauthor = {Jop Briet and Carlos Palazuelos},
  plaintexttitle = {Failure of the trilinear operator space Grothendieck Theorem}, %%  title without math or LaTeX
  keywords = {Grothendieck theorem, operator spaces, additive combinatorics},
}   %%% END \dajAUTHORdetails

\dajEDITORdetails{%
   year={2019},
   number={8},
   received={18 January 2019},   % received date: example: 7 January 2017
   revised={19 April 2019},    % Optional revised date (you may comment it out)
   published={4 June 2019},  % published date
   doi={10.19086/da.8805},       % XXX = number of paper, e.g. da006 for paper#6
}   %%% END \dajEDITORdetails

\begin{document}

\begin{frontmatter}[classification=text]
%% EDITOR: this will force the keywords to appear right after the Abstract.
%%   If the abstract is too long and would force the keywords off the
%%   front page, please comment out % [classification=text] above
%%   This way the keywords will be floated on the bottom of the first page
%%   even though the Abstract spills over to the next page.

%%% AUTHOR: Title goes here.  This line is optional.  You must use it
%%   if title has footnote attached or requires nontrivial typesetting,
%%   e.g., inclusion of linebreaks to force nice layout.
\title{Failure of the trilinear operator space Grothendieck Theorem} %% please capitalize all significant words

%%% AUTHOR:
%%% List all authors. If you wish, place grant acknowledgements in \thanks.
%%% In brackets include a short tag for each author.
\author[jop]{Jop Bri\"{e}t\thanks{Supported  by  a  VENI  grant  and  the  Gravitation grant  NETWORKS-024.002.003  from  the Netherlands Organisation for Scientific Research (NWO).}}
\author[carlos]{Carlos Palazuelos\thanks{Supported by  the Spanish ``Ram\'on y Cajal program'' (RYC-2012-10449), ``Severo Ochoa Programe'' for Centres of Excellence (SEV-2015-0554) and MEC (grant MTM2017-88385-P)}}

%%% AUTHOR: Abstract goes here
\begin{abstract}
We give a counterexample to a trilinear version of the operator space Grothendieck theorem. In particular, we show that for trilinear forms on~$\ell_\infty$, the ratio of the symmetrized completely bounded norm and the jointly completely bounded norm is in general unbounded,
answering a question of Pisier. The proof is based on a non-commutative version of the generalized von Neumann inequality from additive combinatorics.
\end{abstract}
\end{frontmatter}

%%% AUTHOR: body of paper starts here
\section{Introduction}

In the following, let $\mathcal A,\mathcal B$ be $C^*$-algebras and let $\Phi:\A\times \B\to\C$ be a bilinear form.
The  fundamental  theorem  in  the  metric  theory  of  tensor products, better known as Grothendieck's theorem or GT~\cite{Gr53} implies that  if $\A,\B$ are commutative, then the following holds. There exists a universal constant $K$ such that $\|\Phi\| \leq \|\Phi\|_{\gamma_2} \leq K\|\Phi\|$, where~$\|\Phi\|$ is the operator norm and $\|\Phi\|_{\gamma_2} $ is the factorization norm, which quantifies how well the bilinear form factorizes through the inner product of Hilbert spaces: 
\beqn
\|\Phi\|_{\gamma_2}
=
\inf\{
\|\Psi_1\|\|\Psi_2\|\},
\eeqn
where the infimum is taken over Hilbert spaces~$\HS$ and linear maps $\Psi_1:\A\to\HS,
\Psi_2:\B\to\HS$ such that for any $a\in \A,b\in \B$, we have $\Phi(a,b) = \langle \Psi_1(a),\Psi_2(b)\rangle$.
In the same work, Grothendieck conjectured that the assumption that~$\A,\B$ are commutative is unnecessary. 
This was first proved by Pisier in \cite{Pi78} under some approximation assumptions and later in full generality by Haagerup~\cite{Ha85}. 
These results are not only important to Banach space theory, but also found applications in quantum information theory~\cite{Tsi:1987, CJPP:2015, RegevV12a,AAIKS:2016}, computer science~\cite{KN:2012,NRV:2013,BRS:2017} and combinatorics~\cite{CZ:2017}.

\subsection{The operator space GT}
The fact that $C^*$-algebras have a natural operator space structure~\cite{Pi03} invites the study of Grothendieck's theorem in this context. 
In this setting, the relevant norms are the so-called completely bounded norms, which we introduce below; we refer to~\cite{Pi12} for much more detailed information.
We will identify $M_d(\mathcal A)$, the space of $\mathcal A$-valued $d\times d$ matrices, with $\mathcal A\otimes M_d$ (and similarly for~$\mathcal B$).

The \emph{completely bounded norm} of~$\Phi$ is defined by
\beqn
\|\Phi\|_{\cb}
=
\sup_{d\in \N} \|{\Phi}_d\|,
\eeqn
where ${\Phi}_d:M_d(\A)\times M_d(\B) \to M_{d}$ is the ``lift'' given by
\begin{align*}
 &\Big( \sum_i a_i\otimes X_i, \sum_j b_j\otimes Y_j\Big) \mapsto\sum_{i,j} \Phi(a_i,b_j)X_i Y_j.
\end{align*}
The \emph{jointly completely bounded norm} of~$\Phi$ is defined by
\beqn
\|\Phi\|_{\jcb}= \sup_{d\in \N}\|\widetilde\Phi_d\|,
\eeqn
where $\widetilde\Phi_d:M_d(\A)\times M_d(\B) \to M_{d^2}$ is  given by
\begin{align*}
 &\Big( \sum_i a_i\otimes X_i, \sum_j b_j\otimes Y_j\Big) \mapsto\sum_{i,j} \Phi(a_i,b_j)X_i\otimes Y_j.
\end{align*}

It is easy to see that $\|\Phi\|_{\jcb}\leq\|\Phi\|_{\cb}$ (consider the operators $X_i\otimes \Id$ and 
$\Id\otimes Y_j$ when computing~$\|\Phi_d\|$).
It follows from Grothendieck's theorem that if~$\mathcal A,\mathcal B$ are commutative $C^*$-algebras, then these norms are equivalent.
However, their ratio is unbounded in general.
An important difference between these two norms is that only the second is commutative, by which we mean the following.
Define $\Phi^{\mathsf T}:\B\times \A\to \C$ by $\Phi^{\mathsf T}(b,a) = \Phi(a, b)$. 
Then, the jointly completely bounded norm is invariant with respect to this operation, but the completely bounded norm generally is not. 
The following ``symmetrized'' version of the completely bounded norm, introduced in~\cite{OP99}, is again commutative in this sense:
\beqn
\|\Phi\|_{\sym}
=
\inf\big\{\|\Psi_1\|_{\cb} + \|\Psi_2^{\mathsf T}\|_{\cb} \st  \Phi = \Psi_1 + \Psi_2\big\},
\eeqn 
where the infimum is over bilinear forms $\Psi_1,\Psi_2:\A\times \B\to \C$.
It turns out that this norm is equal to an operator space version of the factorization norm mentioned above, provided Hilbert spaces are endowed with the right operator space structure~\cite[Section~18]{Pi12}.
It still holds that $\|\Phi\|_{\jcb}\leq \|\Phi\|_{\sym}$.
Indeed, it follows from the above that
for any decomposition $\Phi = \Psi_1 + \Psi_2$, we have $$\|\Phi\|_{\jcb}\leq \|\Psi_1\|_{\jcb}+\|\Psi_2\|_{\jcb}=\|\Psi_1\|_{\jcb}+\|\Psi_2^{\mathsf T}\|_{\jcb}\leq \|\Psi_1\|_{\cb}+\|\Psi_2^{\mathsf T}\|_{\cb}.$$

Pisier and Shlyakhtenko~\cite{PS02} proved that under certain conditions on the $C^*$-algebras,  the jointly completely bounded norm and the symmetrized completely bounded norm are equivalent, giving an operator space version of Grothendieck's theorem. This result was refined by Haagerup and Musat~\cite{HM08} showing the following result. 

\begin{theorem}[Operator space GT]\label{thm:osgt}
Let $\mathcal A,\mathcal B$ be $C^*$-algebras and let $\Phi:\A\times \B\to \C$ be a bilinear form. Then, $\|\Phi\|_{\jcb} \leq \|\Phi\|_{\sym} \leq 2\|\Phi\|_{\jcb}$.
\end{theorem}

\subsection{Trilinear operator space GT}
A natural question is whether Theorem~\ref{thm:osgt} generalizes to trilinear forms. In particular, Pisier~\cite[Problem~21.3]{Pi12} asked the following:
Let $\mathcal{A}_1$, $\mathcal {A}_2$, $\mathcal {A}_3$ be $C^*$-algebras and let $S_3$ denote the permutation group on~$\{1,2,3\}$.
For a trilinear form $\Phi:\A_1\times\A_2\times\A_3\to \C$ and $\pi\in S_3$, define the trilinear form $\Phi\circ\pi:\A_{\pi(1)}\times\A_{\pi(2)}\times\A_{\pi(3)}\to\C$ by
\beq\label{eq:Phipi}
\Phi\circ\pi(a_{\pi(1)},a_{\pi(2)},a_{\pi(3)}) = \Phi(a_1,a_2,a_3).
\eeq
Define
\beqn
\|\Phi\|_{\sym}
=
\inf\Big\{\sum_{\pi\in S_3}\|\Psi_\pi\circ\pi\|_{\cb} \st  \Phi = \sum_{\pi\in S_3}\Psi_\pi\Big\},
\eeqn
where the infimum is over trilinear forms $\Psi_\pi: \mathcal{A}_1\times \mathcal{A}_2\times \mathcal{A}_3\to\C$ indexed by~$S_3$.
Define $\|\Phi\|_{\jcb}$ in the obvious way, using three-fold tensor products.
Then, is it true that  $\|\Phi\|_{\jcb} \leq \|\Phi\|_{\sym}\leq K\|\Phi\|_{\jcb}$ for some absolute constant~$K\in (0,\infty)$?

This question was originally formulated by asking if any trilinear form $\Phi:\mathcal{A}_1\times \mathcal{A}_2\times \mathcal{A}_3\to \C$ that is \emph{jointly completely bounded}, which is to say that $\|\Phi\|_{\jcb}<\infty$, is always \emph{completely bounded}, which is to say that $\|\Phi\|_{\sym}<\infty$. However, this formulation is equivalent by the Open Mapping Theorem.

Here we answer this question in the negative.
In particular, we show that such an inequality can fail already in the simplest possible scenario; that is for commutative~$C^*$-algebras.

\begin{theorem}\label{thm:main}
There exist absolute constants $C>0$ and $c > 0$ such that the following holds.
For every $n\in \N$, there exists a trilinear form $\Phi:\ell_\infty^n\times\ell_\infty^n\times\ell_\infty^n\to \C$ such that $\|\Phi\|_{\sym} \geq C n^c\|\Phi\|_{\jcb}$.
\end{theorem}

Other trilinear versions of Grothendieck's theorem have already been shown to fail in the past.
Smith~\cite{Smi:1988} gave counterexamples to hoped-for trilinear versions of a Grothendieck-type theorem for completely bounded bilinear forms on $C^*$-algebras due to Pisier~\cite{Pi78} and Haagerup~\cite{Ha85}.
Blecher~\cite{Ble:1989} introduced the notion of tracially completely bounded multilinear forms. These maps form a subspace strictly contained in the space of completely bounded multilinear forms.
It was shown there that bounded bilinear forms on $C^*$-algebras are always tracially completely bounded, which may be interpreted as another Grothendieck-type theorem, but that this is false for trilinear forms in general.
However, these works did not concern the jointly completely bounded norm, which is the appropriate norm in the context of operator spaces.
In~\cite{PGWP+08} it was shown that the operator norm and the jointly completely bounded norm are not equivalent for trilinear forms on commutative $C^*$-algebras (proving the existence of bounded trilinear forms which are not jointly completely bounded). This can be understood as a failure of yet another version of Grothendieck's theorem. The main result in~\cite{PGWP+08}  was later quantitatively improved in~\cite{BV13}, but the optimal ratio between these norms as a function of the dimension is still an open problem.

Remarkably, both the jointly and symmetrized completely bounded norms again turn out to play an important role in quantum information theory.  While the first appears naturally in the context of tri-partite entanglement, in particular as the quantum bias of three-player XOR games (or equivalently, the quantum value of a tripartite correlation Bell inequality)~\cite{PV:2016}, the second norm was recently used in the context of quantum algorithms, to give a characterization of quantum query complexity~\cite{ABP:2018}. 
In a sense, Theorem \ref{thm:main} can be also read as an absence of a direct connection between these topics.

The proof of Theorem~\ref{thm:main} uses a non-commutative version of the generalized von Neumann inequality from additive combinatorics.
This inequality allows us to upper bound the jointly completely bounded norm of certain structured trilinear forms, given by a function~$f$ on a finite Abelian group~$\Gamma$, by the Gowers 3-uniformity norm of~$f$.
An argument of Varopoulos can be used to show that the symmetrized completely bounded norm of such  trilinear forms is always at least~$|\Gamma| \|f\|_{\ell_2}^2$. A simple explicit choice of a function~$f$ from the group~ $\Zp:=\Z/p\Z$ for  prime $p \geq 5$ to the complex unit circle gives the result with $c = 1/8$.
This follows from an elementary Weil-type exponential sum estimate used to upper bound the Gowers 3-uniformity norm of~$f$ by $(2/p)^{1/8}p^2$, while the Varopoulos bound shows that the symmetrized completely bounded norm is at least~$p^2$. In the last section we comment on possible modifications of our construction.

\section{Proof of Theorem~\ref{thm:main}}

\subsection{Preliminaries}
We use the following notational conventions and  basic facts.
Denote $[n] = \{1,\dots,n\}$ and $\T = \{w\in \C \st |w| = 1\}$.
For a set~$S$ let $(e_s)_{s\in S}$ be the standard basis for~$\C^S$.
Below, the set~$S$ will vary but will be clear from the context.
Let ${B_{M_d} = \{X\in M_d \st \|X\| \leq 1\}}$, where $\|X\|$ denotes the usual operator norm on~$M_d$.
 Recall that the commutator of $X,Y\in M_d$ is defined by $[X,Y] = XY - YX$ and that $X,Y$ are said to commute if their commutator is zero.
We will  use the standard notation~$\ell_\infty^n$ for the $n$-dimensional commutative $C^*$-algebra given by~$\C^n$ endowed with the sup norm and coordinate-wise multiplication.
We refer to a trilinear form~$\Phi:\ell_\infty^n\times\ell_\infty^n\times\ell_\infty^n\to \C$ as a trilinear form on~$\ell_\infty^n$.

Note that~$\ell_\infty^n$ can be identified with the space of $n\times n$ diagonal matrices endowed with the operator norm. 
In turn, this implies that $M_d(\ell_\infty^n)$ can be identified with the space of maps $X:[n]\to M_d$, where~$X(i)$ corresponds to the $i$th diagonal block of an element in~$M_d(\ell_\infty^n)$.
As such, the unit ball of $M_d(\ell_\infty^n)$ consists of the maps~$X$ such that $X(i)\in B_{M_d}$ for all $i\in[n]$. Then, it is not hard to see that
\beq\label{cb-norm}
\|\Phi\|_{\cb}  = \sup\big\{\big\|\Phi_d(X,Y,Z)\|_{M_d} \st d\in \N, \: X,Y,Z:[n]\to B_{M_d}\big\},
\eeq
where
\beq\label{eq:Phid}
\Phi_d(X,Y,Z) = \sum_{i,j,k=1}^n\Phi(e_i,e_j,e_k)X(i)Y(j)Z(k).
\eeq
Note that if $\Phi = \sum_{\pi\in S_3}\Psi_\pi$ for some trilinear forms $\Psi_\pi$, then this decomposition holds also for the ``lifts'':
$\Phi_d = \sum_{\pi\in S_3}(\Psi_\pi)_d$.
Similarly,
\beq\label{jcb-norm}
\|\Phi\|_{\jcb}  = \sup\big\{\big\|\widetilde\Phi_d(X,Y,Z)\|_{M_{d^3}} \st  d\in \N, \: X,Y,Z:[n]\to B_{M_d}\big\},
\eeq
where
\beq\label{eq:Phidtilde}
\widetilde\Phi_d(X,Y,Z)= \sum_{i,j,k=1}^n\Phi(e_i,e_j,e_k)X(i)\otimes Y(j)\otimes Z(k).
\eeq

\subsection{The example}

Let~$\Gamma$ be a finite Abelian group and $f_0:\Gamma\to\C$ be some function. 
 Identify $\ell_\infty^\Gamma$ with the function space~$L_\infty(\Gamma)$. 
Define the trilinear form  $\Phi:L_\infty(\Gamma)\times L_\infty(\Gamma)\times L_\infty(\Gamma)\to \C$ by
\beq\label{def:APform}
\Phi(f_1,f_2,f_3) = \sum_{x,y\in \Gamma} f_0(y)f_1(x)f_2(x+y)f_3(x+2y).
\eeq
Theorem~\ref{thm:main} is based on a form as above, for the group~$\Zp$ with prime~$p \geq 5$. To get an example for arbitrary integer~$n\geq 4$, one can choose an odd prime between~$n/2$ and~$n$ (which exists by Bertrand's postulate) and embed~$\Phi$ as in~\eqref{def:APform} based on this group into a trilinear form on~$\ell_\infty^n$ in the obvious way.
In the following two subsections we upper and lower bound the jointly completely bounded norm and the symmetrized completely bounded norm, respectively.

\subsection{Bounding the jointly completely bounded norm}

Let~$\Phi$ be a trilinear form as in~\eqref{def:APform}.
We bound its jointly completely bounded norm using a non-commutative version of the generalized von Neumann inequality. The scalar version of this inequality, a basic tool in additive combinatorics, shows that the operator norm of~$\Phi$ can be bounded from above in terms of the Gowers uniformity norm of~$f_0$.
It was observed already in~\cite{BBBLL:2018} that this inequality holds also for the jointly completely bounded norm; in fact, they prove a more general version than what we use here.
Here, we give an alternative short proof---a straightforward adaptation of the standard proof of the scalar case~\cite[Chapter~11]{TV06}---for the version that is sufficient for our purpose.
To state the inequality, we first define the Gowers uniformity norms (we refer to~\cite{TV06} for more information on these norms).

For a finite set~$S$, denote
\beqn
\Exp_{s\in S}[f(s)] = \frac{1}{|S|}\sum_{s\in S}f(s).
\eeqn

\begin{definition}[Gowers uniformity norms]
Let $k$ be a positive integer, let $\Gamma$ be a finite Abelian group and $f:\Gamma\to \C$ be some function.
Then, the Gowers $U^k$-norm of~$f$ is given by
\beqn
\|f\|_{U^t(\Gamma)}=\big|
\Exp_{x,h_1,\dots,h_k}\big[
(\Delta_{h_1}\cdots\Delta_{h_k}f)(x)
\big]\big|^{\frac{1}{2^k}},
\eeqn
where $\Delta_h f(x) = \overline{f(x)}f(x+h)$.
\end{definition}

The case $k=1$ is strictly speaking not a norm but it is a seminorm.
As an example, the 8th power of the Gowers $U^3$-norm is given by
\begin{multline}
\big|\Exp\big[\overline{f(x)}f(x+h_1)f(x+h_2)f(x+h_3)\overline{f(x+h_1+h_2)}\times\\
\overline{f(x+h_1+h_3)f(x+h_2+h_3)}f(x+h_1+h_2+h_3)\big]\big|,\label{eq:U3}
\end{multline}
where the expectation is over independent uniform $x,h_1,h_2,h_3\in \Gamma$.

Our upper bound on~$\|\Phi\|_{\jcb}$ is based on the following inequality.

\begin{proposition}\label{prop:matrixgvn}
Let~$\Gamma$ be a finite Abelian group and let~$f_0: \Gamma\to \C$ be some function. 
Then, for~$\Phi$ as in~\eqref{def:APform}, we have
\beqn
\|\Phi\|_{\jcb} \leq |\Gamma|^{2}\,\|f_0\|_{U^3(\Gamma)}.
\eeqn
\end{proposition}

To prove this result, let us introduce the following non-commutative version of the Gowers uniformity norms.

\begin{definition}
For positive integes~$d,k$, a finite Abelian group~$\Gamma$ and function $F:\Gamma \to M_d$, define
\beqn
\|F\|_{U^k(\Gamma)}
=
\big\| \Exp_{x,h_1,\dots,h_{k}}(\Delta_{h_1}\cdots\Delta_{h_{k}}F)(x)\big \|^{\frac{1}{2^{k}}},
\eeqn
where (with abuse of notation) $(\Delta_h F)(x) = F(x)^*F(x+h)$.
\end{definition}

\begin{remark}
In general it appears to be unknown if these functions are also norms (but for our purposes we do not need them to be).
Expressions related to the case $k = 2$ were studied in works of Gowers and Hatami~\cite{GowersH:2017} and Chiffre, Ozawa and Thom~\cite{ChiffreOT:2019}.
\end{remark}

Proposition~\ref{prop:matrixgvn} follows from the following key lemma, which is a non-commutative version of the generalized von Neumann theorem.

\begin{lemma}\label{lem:matrixgvN2}
Let  $d\in \N$ and let~$\Gamma$ be a finite Abelian group. 
Let $A_0, A_1,A_2,A_3:\Gamma \to B_{M_d}$ be maps such that for all $x,y\in \Gamma$ and distinct $i,j=0,1,2,3$, we have $[A_i(x),A_j(y)]=[A_i(x)^*,A_j(y)]=0$. Then,
\beqn\label{eq:gvN}
\big\|\Exp_{x,y\in \Gamma}A_0(y)A_1(x) A_2(x+y)A_3(x+2y) \big\|
\leq
\|A_{0}\|_{U^{3}(\Gamma)}.
\eeqn
\end{lemma}

A version of Lemma~\ref{lem:matrixgvN2} with $k$-term arithmetic progressions instead of 3-term arithmetic progressions also holds with the right-hand side replaced with $\|A_0\|_{U^k(\Gamma)}$.
More generally, other known variants of the scalar case hold also in this non-commutative setting.
The proof of Lemma~\ref{lem:matrixgvN2} uses the following ``matrix van der Corput lemma''.

\begin{lemma}\label{lem:vdCorput}
Let $\Gamma$ be a finite Abelian group, let $S$ be a finite set and for each $s\in S$ let $F_s: \Gamma\to M_d$. 
Then, for any map ${\bf B}: S\to B_{M_d}$, 
\beqn
\big\|\Exp_{s \in S}\Exp_{x\in \Gamma}{\bf B}(s)F_s(x)\big\|
\leq
\big\|
\Exp_{s\in S}\Exp_{x,h\in \Gamma}(\Delta_h F_s)(x)
\big\|^{\frac{1}{2}}.
\eeqn
\end{lemma}

\begin{proof}
Let $F(s) = \Exp_{x\in \Gamma}F_s(x)$.
The Cauchy--Schwarz inequality and boundedness of ${\bf B}$ give
\begin{align*}
\big\|\Exp_{s \in S}\Exp_{x\in \Gamma} {\bf B}(s)F_s(x)\big\|&=
\big\|\Exp_{s\in S}{\bf B}(s)F(s)\big\|\\
&\leq
\big\|\Exp_{s\in S} {\bf B}(s){\bf B}(s)^*\big\|^{\frac{1}{2}}
\big\|\Exp_{s\in S} F(s)^*F(s)\big\|^{\frac{1}{2}}\\
&\leq
\big\|\Exp_{s\in S} \Exp_{x,y\in \Gamma}F_s(x)^*F_s(y)\big\|^{\frac{1}{2}}.
\end{align*}
The claim now follows by substituting $y = x+h$.
\end{proof}

\begin{proof}[ of Lemma~\ref{lem:matrixgvN2}]
We will repeatedly use the fact that the map $(x,y)\mapsto (x-y,y)$ on $\Gamma\times\Gamma$ is bijective.
The proof uses Lemma~\ref{lem:vdCorput} three times, with different choices of $S, {\bf B}$ and $F_s$.

First, let $S = \Gamma$, let ${\bf B} = A_1$ and let $F_x(y) = A_2(x+y)A_3(x+2y)A_0(y)$.
Then Lemma~\ref{lem:vdCorput} and commutativity of the $A_i$ give
\begin{align}
\big\|\Exp_{x,y} A_0(y)A_1(x) A_2(x+y)A_3(x+2y) \big\|^8
&=
\big\|\Exp_{x}{\bf B}(x)\Exp_yF_x(y)\big\|^8\nonumber\\
&\leq
\big\|
\Exp_{x}\Exp_{y,h_1}(\Delta_{h_1}F_x)(y)
\big\|^{4}.
\label{eq:pf1}
\end{align}
Using the above-mentioned change of variables, the right-hand side of~\eqref{eq:pf1} equals 
\beq\label{eq:pf1b}
\big\|\Exp_{x,h_1}\Exp_y F_{x-y}(y)^* F_{x-y}(y+h_1)\big\|^4.
\eeq

Second, using the properties of the maps~$A_2,A_3,A_0$, it follows that
\beqn
F_{x-y}(y)^* F_{x-y}(y+h_1) =
A_2(x)^*A_2(x + h_1)A_3(x+y)^*A_3(x+y+2h_1)\, (\Delta_{h_1}A_0)(y).
\eeqn
Now let $S=\Gamma\times \Gamma$ and factor the above as
\begin{align*}
{\bf B}(x,h_1) &= A_2(x)^*A_2(x + h_1)\\
F_{x,h_1}(y) &= A_3(x+y)^*A_3(x+y+2h_1)\, (\Delta_{h_1}A_0)(y).
\end{align*}
From Lemma~\ref{lem:vdCorput} and another change of variables, it follows that the right-hand side of~\eqref{eq:pf1b} is at most
\beq\label{eq:pf2}
\big\|
\Exp_{x,h_1}\Exp_{y, h_2}(\Delta_{h_2}F_{x,h_1})(y)
\big\|^2
=
\big\|
\Exp_{x,h_1,h_2}\Exp_{y}F_{x-y,h_1}(y)^*F_{x-y,h_1}(y+h_2)
\big\|^2.
\eeq

Third, it follows from the properties of $A_3, A_0$ that
\beqn
F_{x-y,h_1}(y)^*F_{x-y,h_1}(y+h_2)
=A_3(x+2h_1)^*A_3(x)A_3(x+h_2)^* A_3(x+2h_1+h_2) (\Delta_{h_2}\Delta_{h_1}A_0)(y).
\eeqn
Finally set $S=\Gamma\times \Gamma \times \Gamma$ and factor the above as
\begin{align*}
{\bf B}(x,h_1,h_2) &=A_3(x+2h_1)^*A_3(x)A_3(x+h_2)^*A_3(x+2h_1+h_2)\\
F_{x,h_1,h_2}(y) &= (\Delta_{h_2}\Delta_{h_1}A_0)(y).
\end{align*}
Again by Lemma~\ref{lem:vdCorput}, the right-hand side of~\eqref{eq:pf2} is at most
\beqn
\big\|\Exp_{x,h_1,h_2}\Exp_{y,h_3}(\Delta_{h_3}F_{x,h_1,h_2})(y)\big\|
=
\big\|
\Exp_{h_1,h_2,h_3,y}(\Delta_{h_3}\Delta_{h_2}\Delta_{h_1}A_0)(y)
\big\|,
\eeqn
giving the claim.
\end{proof}

\begin{proof}[ of Proposition ~\ref{prop:matrixgvn}]
For any $d\in \N$ and
 $X,Y,Z:\Gamma\to B_{M_d}$, define $A_0,A_1,A_2,A_3: \Gamma\to B_{M_{d^3}}$ by
 \begin{align*}
A_0(x) &=f_0(x) \Id\otimes \Id \otimes \Id\\
A_1(x) &=X(x)\otimes \Id \otimes \Id\\
A_2(x) &=\Id\otimes Y(x) \otimes \Id\\
A_3(x) &=\Id\otimes \Id \otimes Z(x).
\end{align*}
Then, the statement follows trivially from Lemma~\ref{lem:matrixgvN2}  and noting that the factor $|\Gamma|^{2}$ comes from replacing sums with expectations.
\end{proof}

\begin{remark}
Note that Lemma~\ref{lem:matrixgvN2} also applies if $M_d$ is replaced by the space $B(\mathcal H)$ of bounded operators on a (possibly infinite-dimensional) Hilbert space~$\mathcal H$. Moreover, the upper bound stated in Proposition~\ref{prop:matrixgvn} even applies if one replaces the jointly completely bounded norm by
\begin{align*}
&\|\Phi\|_{\mathrm{com}}:=
\sup\Big\{\Big\| \sum_{i,j,k=1}^n \Phi(e_i,e_j,e_k) X_1(i)X_2(j)X_3(k)\Big\|_{\mathcal B(\mathcal H)}\Big\},
\end{align*}
where the supremum is over maps $X_1,X_2,X_3:[n] \to B_{\mathcal B(\mathcal H)}$ such that $[X_i(k),X_j(l)]=0$ and $[X_i(k)^*,X_j(l)]=0$ for all $k,l\in [n]$ and $i\neq j$. 
\end{remark}

\begin{proposition}\label{prop:weyl}
Let $p\geq 5$ be a prime,  $\omega = e^{2\pi i/p}$ and $f_0:\Zp\to \C$ be the function given by $f_0(x) = \omega^{x^3}$.
Then,
\beqn
\|f_0\|_{U^3(\Zp)} \leq (2/p)^{1/8}.
\eeqn
\end{proposition}

\begin{proof}
A straightforward calculation shows that for $x,h_1,h_2,h_3\in \Z_p$, we have
\beqn
(\Delta_{h_1}\Delta_{h_2}\Delta_{h_3}f_0)(x)
=
\omega^{6 h_1h_2h_3}.
\eeqn
It follows that
\begin{align*}
\|f_0\|_{U^3(\Zp)}^{8}
=
\Exp_{h_1,h_2\in\Z_p}\big[\Exp_{h_3\in\Zp}[\omega^{6h_1h_2h_3}]\big].
\end{align*}
The inner expectation over~$h_3$ is~1 if $h_1=0$ or $h_2=0$ and, since~$6$ is coprime relative to~$p$ and $\Z_p$ is a field, it is~0 otherwise.
Hence, the right-hand side equals $(2p-1)/p^2$, which gives the claim.
\end{proof}

\begin{corollary}\label{cor:weylPhi}
Let $p\geq 5$ be a prime, let $\Gamma = \Zp$ and let $f_0$ be as in Proposition~\ref{prop:weyl}.
Then, for~$\Phi$ as in~\eqref{def:APform}, we have
\beqn
\|\Phi\|_{\jcb} \leq p^2(2/p)^{1/8}.
\eeqn
\end{corollary}

\subsection{Bounding the symmetrized completely bounded norm}

To lower bound the symmetrized completely bounded norm, we first prove the following result.

\begin{lemma}\label{Lower commuting}
Let $\Psi$ be a trilinear form on~$\ell_\infty^n$.
Then,
\beqn
\|\Psi\|_{\sym} \geq \sup\big\{\|\Psi_d(X_1,X_2,X_3)\|_{M_d}\st d\in \N, \: X_1,X_2,X_3:[n]\to B_{M_d}\big\},
\eeqn
where the supremum is over maps $X_i$ such that $[X_i(k),X_j(l)]=0$ for all $k,l\in [n]$ and $i\neq j$.\footnote{Note that in contrast with the norm $\|\Psi\|_{\mathrm{com}}$ defined above, here we do not require that $[X_i(k)^*,X_j(l)] = 0$.}
\end{lemma}

This result was already proved in \cite{OP99} in much greater generality and the authors showed that the quantities appearing in Lemma~\ref{Lower commuting} are equivalent. 
Since the proof of the inequality we need is straightforward, we add it for completeness. 

\begin{proof}
Let $\Psi =\sum_{\pi \in S_3} \Psi_{\pi}$ be some decomposition.
Let~$d$ be a positive integer and let $X_1,X_2,X_3:[n]\to M_d$ be maps with commuting ranges as in the lemma. By the triangle inequality,
\begin{align}\label{eq:pisum}
\|\Psi_d(X_1,X_2,X_3)\|_{M_d} 
&\leq
\sum_{\pi\in S_3}\| (\Psi_\pi)_d(X_1,X_2,X_3)\|_{M_d}.
\end{align}
We claim that each term on the right-hand side equals
\beqn
\|(\Psi_{\pi}\circ\pi)_d(X_{\pi(1)}, X_{\pi(2)}, X_{\pi(3)})\|_{M_d}.
\eeqn
This implies the lemma because the above is clearly at most $\|\Psi_{\pi}\circ\pi\|_{\cb}$.
To see the claim, first observe that by commutativity, it holds that for every $i_1,i_2,i_3\in[n]$ and $\pi\in S_3$, we have
\beq\label{eq:Xcomm}
X_1(i_1)X_2(i_2)X_3(i_3)
=
X_{\pi(1)}(i_{\pi(1)}) X_{\pi(2)}(i_{\pi(2)}) X_{\pi(3)}(i_{\pi(3)}).
\eeq
Let $\chi$ be some trilinear form on~$\ell_\infty^n$.
Recall from~\eqref{eq:Phipi} that
\beq\label{eq:chipi}
\chi(e_{i_{\pi^{-1}(1)}},e_{i_{\pi^{-1}(2)}},e_{i_{\pi^{-1}(3)}}) = (\chi\circ\pi)(e_{i_1},e_{i_2},e_{i_3}).
\eeq
Then, 
\begin{align*}
\chi_d(X_1,X_2,X_3)
&=
\sum_{i_1,i_2,i_3=1}^n \chi(e_{i_1},e_{i_2},e_{i_3})X_1(i_1)X_2(i_2)X_3(i_3)\\
&\stackrel{\eqref{eq:Xcomm}}{=}
\sum_{i_1,i_2,i_3=1}^n \chi(e_{i_1},e_{i_2},e_{i_3})X_{\pi(1)}(i_{\pi(1)}) X_{\pi(2)}(i_{\pi(2)}) X_{\pi(3)}(i_{\pi(3)})\\
&\stackrel{\eqref{eq:chipi}}{=}
\sum_{i_1,i_2,i_3=1}^n (\chi\circ\pi)(e_{i_1},e_{i_2},e_{i_3})X_{\pi(1)}(i_1)X_{\pi(2)}(i_2)X_{\pi(3)}(i_3)\\
&=
(\chi\circ\pi)_d(X_{\pi(1)},X_{\pi(2)},X_{\pi(2)}).
\end{align*}
Applying this to $\chi = \Psi_\pi$ for each $\pi$ gives the claim.
\end{proof}

A trilinear form~$\Psi$ on~$\C^n$ is \emph{symmetric} if $\Psi\circ\pi = \Psi$ holds for every $\pi\in S_3$.
A \emph{slice} of a (not necessarily symmetric) trilinear form $\Psi$  is an $n\times n$ matrix obtained by fixing one of the three coordinates (so there are $3n$ slices), for example
\beqn
M_i=(\Psi(e_i, e_j, e_k))_{j,k=1}^n.
\eeqn
We will denote
\beqn
\Delta(\Psi)
=
\max\{\|M\| \st \text{$M$ is a slice of $\Psi$}\}.
\eeqn
Also define
\beqn
\|\Psi\|_{\ell_2} = \Big(\sum_{i,j,k=1}^n|\Psi(e_i,e_j,e_k)|^2\Big)^{\tfrac{1}{2}}.
\eeqn

The following lemma, due to Varopoulos~\cite{Va74}, is the key to our lower bound on $\|\Psi\|_{\sym}$. Again, the proof is simple, so we add it for completeness.

\begin{lemma}\label{lem:varo}
Let~$\Psi$ be a symmetric trilinear form on~$\ell_\infty^n$.
Then,
\beqn
\|\Psi\|_{\sym}
\geq
\frac{\|\Psi\|_{\ell_2}^2}{\Delta(\Psi)}.
\eeqn
\end{lemma}
\begin{proof}
For each $i\in [n]$, let $M_i=(\Psi(e_i,e_j, e_k))_{j,k=1}^n$ be the slice obtained by fixing the first coordinate to~$i$.
Define $W_i = \Delta(\Psi)^{-1}M_i$ and note that this has operator norm at most~1.
For each $i\in[n]$, define the $(2n+2)\times(2n+2)$ block matrix
\beqn
X_i
=
{\footnotesize
\left[
\begin{array}{c|c|c|c}
\phantom{M}&&&\\
\hline
e_i&\phantom{M}&&\\
\hline
&W_i^*&\phantom{M}&\\
\hline
&&e_i^*&\phantom{M}
\end{array}
\right]},
\eeqn
where the row and column blocks have size $1,n,n,1$, respectively, and
where the empty blocks are filled with zeros.
Then, for all $i,j\in[n]$,
\beqn
X_i^*X_i
=
{\footnotesize
\left[
\begin{array}{c|c|c|c}
1&&&\\
\hline
&W_iW_i^*&&\\
\hline
&&e_ie_i^*&\\
\hline
&&&\phantom{x}
\end{array}
\right]}
\quad
\text{and}
\quad
X_iX_j
=
{\footnotesize
\left[
\begin{array}{c|c|c|c}
\phantom{x}&&&\\
\hline
&\phantom{x}&&\\
\hline
W_i^*e_j&&\phantom{x}&\\
\hline
&e_i^*W_j^*&&\phantom{x}
\end{array}
\right]}.
\eeqn
The first identity shows that $\|X_i\| = \max\{1,\|W_i\|\} \leq 1$.
Since $\Psi$ is symmetric, we have $M_je_i = M_ie_j$ for all $i,j$.
Therefore, the second identity shows that these matrices commute with each other.
Moreover,
\beqn
X_iX_jX_k
=
{\footnotesize
\left[
\begin{array}{c|c|c|c}
\phantom{x}&&&\\
\hline
&\phantom{x}&&\\
\hline
&&\phantom{x}&\\
\hline
e_i^*W_j^*e_k&&&\phantom{x}
\end{array}
\right]
=
\frac{1}{\Delta(\Psi)}
\left[
\begin{array}{c|c|c|c}
\phantom{x}&&&\\
\hline
&\phantom{x}&&\\
\hline
&&\phantom{x}&\\
\hline
\overline{\Psi(e_i,e_j, e_k)}&&&\phantom{\hat{\Psi}}
\end{array}
\right].
}
\eeqn

Hence, by Lemma~\ref{Lower commuting}, we get that
\begin{align*}
\|\Psi\|_{\sym}&
\geq 
\Big\| \sum_{i,j,k=1}^n \Psi(e_i,e_j, e_k)X_iX_jX_k\Big\|_{M_d} \geq 
  \frac{1}{\Delta(\Psi)}\sum_{i,j,k=1}^n|\Psi(e_i,e_j, e_k)|^2.
\end{align*} 
This concludes the proof.
\end{proof}

Below we present a self-contained proof of Theorem \ref{thm:main}, so that no prior knowledge of operator space theory is needed. 
But some of the facts we use can be proved faster based on some well-known---albeit somewhat non-trivial---facts from it. 
We briefly outline why this is the case.
Readers not familiar with this theory can safely skip the next few paragraphs and continue at Proposition~\ref{prop:deltaphi}.

Lemma \ref{lem:varo} also follows from the fact that for any trilinear forms $\Psi,\Phi$ on~$\ell_\infty^n$, we have 
\begin{align}\label{remark-reviewer}
\|\Psi\|_{\sym}\Delta (\Phi)\geq |\langle  \Psi, \Phi\rangle|
:=
\Big|\sum_{i,j,k=1}^n\Psi(e_i,e_j,e_k)\overline{\Phi(e_i,e_j,e_k)}\Big|.
\end{align}

We sketch the proof of (\ref{remark-reviewer}).
The identities $\|id:\ell_1^n\rightarrow R_n\|_{\cb}=1$ and $\|id:\ell_1^n\rightarrow C_n\|_{\cb}=1$, where $R_n,C_n$ are the row and column operator spaces respectively, and  $R_n\otimes_h\ell_1^n\otimes_h C_n=S_1^n(\ell_1^n)=\ell_1^n(S_1^n)$  (see for instance \cite[Corollary~5.11]{Pi03}),  imply that the linear map 
$T:\ell_1^n\otimes_h\ell_1^n\otimes_h \ell_1^n\rightarrow \ell_1^n(S_1^n)$ defined by $$T(e_i\otimes e_j\otimes e_k)=e_j\otimes (e_i\otimes e_k),$$ is a (complete) contraction.

Then, using that the dual space of $\ell_1^n(S_1^n)$ is the space $\ell_\infty^n(S_\infty^n)$, one gets that for $\Psi,\Phi$ as above,
 $$\|\Psi\|_{\ell_1^n\otimes_h\ell_1^n\otimes_h \ell_1^n}\Delta(\Phi)\geq\|\Psi\|_{\ell_1^n\otimes_h\ell_1^n\otimes_h \ell_1^n}\max_j \|M_j(\Phi)\|_{S_\infty^n}\geq |\langle \Psi,\Phi\rangle|.$$

Hence, since $\Delta(\Phi)$ controls all the indices, it follows that for any permutation $\pi\in S_3$ we have $$\|\Psi\circ \pi\|_{\ell_1^n\otimes_h\ell_1^n\otimes_h \ell_1^n}\Delta(\Phi)\geq |\langle \Psi,\Phi\rangle|,$$from which~\eqref{remark-reviewer} is easily obtained.
Note that based on this argument,~$\Psi$ is not required to be symmetric and as a consequence, Proposition~\ref{prop:symm} below is no longer needed.

\begin{proposition}\label{prop:deltaphi}
Let $p\geq 3$ be a prime, let $\Gamma = \Z_p$, let $f_0:\Gamma\to \T$ and let~$\Phi$ be a trilinear form as in~\eqref{def:APform}.
Then, $\|\Phi\|_{\ell_2}^2=p^2$ and  ${\Delta(\Phi) = 1}$.
\end{proposition}

\begin{proof}
The first assertion is straightforward to check.
Let $\{e_x\st x\in \Gamma\}$ denote the standard basis for~$\C^\Gamma$.
Fix a $x\in \Gamma$ and consider the slice obtained by fixing the first coordinate of the tensor corresponding to~$\Phi$ to $x$:
\begin{align*}
M_x&= \sum_{y,z\in \Gamma}\Phi(e_x,e_y, e_z)e_y\otimes e_z\\
&= \sum_{y,z\in \Gamma}\sum_{u,v\in \Gamma}f_0(v)e_x(u)e_y(u+v)e_z(u+2v) e_y\otimes e_z\\
&= \sum_{y,z\in \Gamma}\sum_{v\in \Gamma}f_0(v)e_y(x+v)e_z(x+2v) e_y\otimes e_z\\
&= \sum_{y,z\in \Gamma}f_0(y-x)e_z(2y-x) e_y\otimes e_z\\
&=\sum_{y\in \Gamma}f_0(y-x)e_y\otimes e_{2y-x}.
\end{align*} Since for our group $\Zp$, the map $y\mapsto 2y$ is injective, it follows that~$M_x$ is a unitary matrix and therefore has norm~$1$.
The other slices can similarly be seen to have norm~$1$.
\end{proof}

\subsection{Putting everything together}

To apply Lemma \ref{lem:varo} we need to symmetrize our form. We  do this so as to approximately preserve  $\Delta(\Phi)$, $\|\Phi\|_{\ell_2}$ and $\|\Phi\|_{\jcb}$.\footnote{Perhaps a more natural symmetrization to consider is $\Phi_s=\sum_{\pi} \Phi\circ \pi$. However, the relevant values can be dramatically affected by this procedure, since we could get a zero tensor from a non-zero one.} To this end, we first consider the trilinear form $E:\C^3\times\C^3\times\C^3\to \C$ given by 
\beqn
E(u,v,w) = u_1v_2w_3.
\eeqn
For a trilinear form $\Psi$ on $\C^n$, 
 the trilinear form $\Psi\otimes E$ on $\C^{3n}$ is given by
\beqn
(\Psi\otimes E)(x\otimes u, y\otimes v,z\otimes w)
=
\Psi(x,y,z)E(u,v,w),
\eeqn
for $x,y,z\in \C^n$ and $u,v,w\in\C^3$.
If $\Psi$ is a trilinear form on~$\ell_\infty^n$, then we define the symmetrized version of~$\Psi$ to be the trilinear form~$\overline\Psi$ on~$\ell_\infty^{3n}$ by
\begin{align}\label{summetrization}
\overline\Psi
=
\sum_{\pi\in S_3} (\Psi\otimes E)\circ\pi.
\end{align}
 It is easy to see that~$\overline{\Psi}$ is symmetric.
Moreover, as per~\eqref{eq:Phipi}, for any $x_i\in \C^n$ and $u_i\in \C^3$ for $i=1,2,3$, we have 
\begin{multline}
\overline\Psi(x_1\otimes u_1, x_2\otimes u_2, x_3\otimes u_3)=\\
\sum_{\pi\in S_3} \Psi(x_{\pi^{-1}(1)}, x_{\pi^{-1}(2)}, x_{\pi^{-1}(3)})E(u_{\pi^{-1}(1)}, u_{\pi^{-1}(2)}, u_{\pi^{-1}(3)})=\\
\sum_{\pi\in S_3} \Psi(x_{\pi(1)}, x_{\pi(2)}, x_{\pi(3)})E(u_{\pi(1)}, u_{\pi(2)}, u_{\pi(3)}).\label{eq:Psibarsum}
\end{multline}

\begin{proposition}\label{prop:symm}
Let~$\Psi$ be a trilinear form on~$\ell_\infty^n$. Then, its symmetrization $\overline{\Psi}$ as in~\eqref{summetrization} satisfies:
\begin{itemize}
\setlength\itemsep{-5pt}
\item $\Delta(\overline\Psi) = \Delta(\Psi)$\\
\item $\|\overline\Psi\|_{\ell_2}^2 = 6\|\Psi\|_{\ell_2}^2$\\
\item $\|\overline\Psi\|_{\jcb}\leq 6\|\Psi\|_{\jcb}$.
\end{itemize}
\end{proposition}

\begin{proof}
We begin with the first item.
The lower bound $\Delta(\overline\Psi) \geq \Delta(\Psi)$ follows easily from the fact that
\beqn
\overline\Psi(x\otimes e_1,y\otimes e_2,z\otimes e_3)
=
\Psi(x,y,z).
\eeqn
By symmetry of~$\overline\Psi$, for the upper bound $\Delta(\overline\Psi) \leq \Delta(\Psi)$, it suffices to show that for any $i\in[n]$ and $a\in[3]$, the slice corresponding to the bilinear form $B:\C^{3n}\times \C^{3n}\to \C$ given by 
\beqn
B(x,y) = 
\overline\Psi(e_i\otimes e_a, x,y)
\eeqn
has operator norm at most~$\Delta(\Psi)$.
Let~$x,y\in \C^{3n}$ be unit vectors.
Write $x = x_1\otimes e_1 + x_2\otimes e_2 + x_3\otimes e_3$ for $x_1,x_2,x_3\in \C^n$ and similarly for~$y$.
Then,
\begin{align*}
|B(x,y)|
&=
\Big|\sum_{b,c=1}^3\overline\Psi(e_i\otimes e_a,x_b\otimes e_b, y_c\otimes e_c)\Big|\\
&=
\Big|\sum_{b,c=1}^3\sum_{\pi\in S_3} \big((\Psi\otimes E)\circ\pi\big)(e_i\otimes e_a,x_b\otimes e_b, y_c\otimes e_c)\Big|\\
&=
\Big|\sum_{\pi\in S_3} \sum_{b,c=1}^3 (\Psi\circ\pi)(e_i,x_b,y_c)\cdot (E\circ \pi)(e_a,e_b,e_c)\Big|.
\end{align*}
Observe that $(E\circ \pi)(e_a,e_b,e_c)$ equals~1 if $a=\pi(1),b=\pi(2),c=\pi(3)$ and~$0$ otherwise.
Hence, the above is at most
\begin{align*}
\sum_{\pi\in S_3\st \pi(1)=a}
 \big|(\Psi\circ\pi)(e_i,x_{\pi(2)},y_{\pi(3)})\big|
 &\leq
 \Delta(\Psi)\,\sum_{\pi\in S_3\st \pi(1)=a}\| x_{\pi(2)}\|\,\|y_{\pi(3)}\|.
\end{align*}
By the Cauchy-Schwarz inequality, the last sum is at most
\beqn
\Big( \sum_{\pi\in S_3\st \pi(1)=a}\| x_{\pi(2)}\|^2\Big)^{\tfrac{1}{2}}
\Big( \sum_{\pi\in S_3\st \pi(1)=a}\| y_{\pi(3)}\|^2\Big)^{\tfrac{1}{2}} \leq 1.
\eeqn
This proves the first item.

The second item is a straightforward calculation.
It follows from~\eqref{eq:Psibarsum} that 
\begin{align*}
\|\overline\Psi\|_{\ell_2}^2&=
\sum_{i_1,i_2,i_3=1}^n\sum_{j_1,j_2,j_3=1}^3\Big|\overline\Psi(e_{i_1}\otimes e_{j_1},e_{i_2}\otimes e_{j_2},e_{i_3}\otimes e_{j_3})\Big|^2\\
&=\sum_{i_1,i_2,i_3=1}^n\sum_{j_1,j_2,j_3=1}^3\Big|\sum_{\pi\in S_3}\Psi(e_{i_{\pi(1)}},e_{i_{\pi(2)}},e_{i_{\pi(3)}})E(e_{j_{\pi(1)}},e_{j_{\pi(2)}},e_{j_{\pi(3)}})\Big|^2.
\end{align*}
Observe that $E(e_{j_{\pi(1)}},e_{j_{\pi(2)}},e_{j_{\pi(3)}}) = 1$ only if $j_1,j_2,j_3\in[3]$ are distinct and that in that case there is a unique~$\pi\in S_3$ for which this holds.
Since for any fixed $\pi\in S_3$ we have
\beqn
\sum_{i_1,i_2,i_3=1}^n|\Psi(e_{i_{\pi(1)}},e_{i_{\pi(2)}},e_{i_{\pi(3)}})|^2
=
\|\Psi\|_{\ell^2}^2,
\eeqn
and there are 6 ways to choose $j_1,j_2,j_3$ distinct, the second item follows.

For the third item,  observe that the jointly completely bounded norm is commutative, which is to say that $\|\Psi\circ \pi\|_{\jcb} = \|\Psi\|_{\jcb}$ for every $\pi\in S_3$.
The claim then follows from the identity $\|\Psi\otimes E\|_{\jcb}=\|\Psi\|_{\jcb}$ and triangle inequality. 
To see the identity, recall the expressions~\eqref{jcb-norm} and~\eqref{eq:Phidtilde} for the jointly completely bounded norm.
Let~$d$ be a positive integer and let $X,Y,Z:[n]\times[3]\to B_{M_d}$. Then,
\begin{multline*}
\sum_{i,j,k=1}^n\sum_{a,b,c=1}^3\Psi(e_i,e_j,e_j)E(e_a,e_b,e_c)
X(i,a)\otimes Y(j,b)\otimes Z(k,c)\\
=
\sum_{i,j,k=1}^n
\Psi(e_i,e_j,e_j)
X(i,1)\otimes Y(j,2)\otimes Z(k,3).
\end{multline*}
Taking norms and suprema over $X,Y,Z$ and $d\in \N$ gives identity.
\end{proof}

With this, the proof of Theorem~\ref{thm:main} is straightforward.

\begin{proof}[ of Theorem~\ref{thm:main}]
Let $p\geq 5$ be a prime number and let $\Gamma =\Z_p$.
Let~$\Phi$ be a trilinear form as in~\eqref{def:APform} and let $f_0:\Gamma\to\C$ be as in Proposition~\ref{prop:weyl}.
Let~$\overline\Phi$ be the symmetrization of~$\Phi$ as in~\eqref{summetrization}.
Then, it follows from Corollary~\ref{cor:weylPhi} and Proposition~\ref{prop:symm} that $\|\overline\Phi\|_{\jcb} \leq 6p^2(2/p)^{1/8}$.
On the other hand, it follows from Lemma~\ref{lem:varo}, Proposition~\ref{prop:deltaphi} and Proposition~\ref{prop:symm} that $\|\overline\Phi\|_{\sym} \geq 6p^2$.
\end{proof}

\section{Alternative tensors}

A straightforward argument based on splitting the tensor associated to the trilinear form~$\Psi$ as in~\eqref{def:APform} into real and complex parts shows that Theorem~\ref{thm:main} holds also for a trilinear form  $\Phi:\ell_\infty^n\times\ell_\infty^n\times\ell_\infty^n\to \C$ whose associated tensor is real, which is to say that $\Phi(e_i,e_j,e_k)\in \R$ for every $i,j,k\in[n]$.
Alternatively, one can directly get such a form by replacing Proposition~\ref{prop:weyl} in our construction with the following statement~\cite[Exercise~11.1.17]{TV06},  giving a random example.

\begin{proposition}
Let $\Gamma$ be a finite Abelian group and $f:\Gamma\to\pmset{}$ be a uniformly random mapping. Then,  $\|f\|_{U^t(\Gamma)} \leq O_t(|\Gamma|^{-1/2^t})$
with probability $1 - o_t(1)$.
\end{proposition}

%%% AUTHOR: optional acknowledgments here
\section*{Acknowledgments} %%  you may comment this out if no Ackno
The authors thank the anonymous referee for pointing out the short abstract operator space proof of Lemma~\ref{lem:varo}.

%%% AUTHOR:
%%% Bibliography goes here. Note that the arXiv cannot process bibtex
%%% or biber bibliographies.  Example of acceptable bibliograpy format:
\bibliographystyle{amsplain}

\providecommand{\bysame}{\leavevmode\hbox to3em{\hrulefill}\thinspace}
\providecommand{\MR}{\relax\ifhmode\unskip\space\fi MR }
\providecommand{\MRhref}[2]{%
  \href{http://www.ams.org/mathscinet-getitem?mr=#1}{#2}
}
\providecommand{\href}[2]{#2}

%% AUTHOR: You can generate such a bibliography from a .bib file by 
%% running pdflatex/bibtex/pdflatex/pdflatex and then pasting the .bbl file
%% between \begin{thebibliography} and \end{bibliography}

%%% AUTHOR: Include a short description of each author following the
%%% structure below. Use the same short tags used previously.  
%%% Use \imageat{} and \imagedot{} instead of "@" and "." in
%%% email addresses-this replaces the symbols with graphics to avoid 
%%% e-mail address harvesting from the .pdf file
\begin{dajauthors}
\begin{authorinfo}[jop]
  Jop Bri\"{e}t\\
  CWI - Centrum Wiskunde \& Informatica\\
  Amsterdam, The Netherlands\\
  j.briet\imageat{}cwi.nl\\
  \url{https://homepages.cwi.nl/~jop/}
\end{authorinfo}
\begin{authorinfo}[carlos]
  Carlos Palazuelos\\
  Dep. An\'{a}lisis Matem\'{a}tico y Matem\'{a}tica Aplicada\\
  Universidad Complutense de Madrid \\
  Madrid, Spain\\
  cpalazue\imageat{}mat.ucm.es\\
  \url{https://www.ucm.es/mathqi/carlos-palazuelos-cabezon}
\end{authorinfo}
\end{dajauthors}

\end{document}